\documentclass[thmsec,nolinenbrs]{dmgt-arxiv}

\usepackage{amsmath}
\usepackage[noend]{algpseudocode}
\usepackage[symbol]{footmisc}

\usepackage{tikz}
\usepackage{tkz-berge}
\tikzstyle{vertex}=[auto=left,circle,fill=black,minimum size=12pt,inner sep=0pt]

\newtheorem{observation}[theorem]{Observation}
\newtheorem{Lemma}[theorem]{Lemma}
\newtheorem{algorithm}[theorem]{Algorithm}
\newtheorem{Corollary}[theorem]{Corollary}

\newcommand{\diam}{\mathrm{diam}}
\newcommand{\bbox}{\mathbin{\Box}}

\newcommand{\parents}{\mathrm{parents}}
\newcommand{\children}{\mathrm{children}}
\newcommand{\Swap}{\mathrm{Swap}}

\newauthor{%
Stephen Finbow}{%
S. Finbow}{%
Department of Mathematics and Statistics\\
St. Francis Xavier University\\
4130 University Avenue, Antigonish, NS, B2G 2W5, Canada}[%
sfinbow@stfx.ca]

\newauthor{%
Christopher M. van Bommel\footnote[2]{Corresponding author.}}{%
C.M. van Bommel }{%
Department of Combinatorics and Optimization\\
University of Waterloo\\
200 University Avenue West, Waterloo, ON, N2L 3G1, Canada}[%
cvanbomm@uwaterloo.ca]

\def\tfn{\footnote[1]{This research was funded by Natural Sciences and Engineering Research Council of Canada grant number 2014-06571.}}
\title{$\gamma$-Graphs of Trees\tfn}[$\gamma$-Graphs of Trees]


\classnbr{05C69}

\begin{document}

\begin{abstract}
For a graph $G = (V, E)$, the $\gamma$-graph of $G$, denoted $G(\gamma) = (V(\gamma), E(\gamma))$, is the graph whose vertex set is the collection of minimum dominating sets, or $\gamma$-sets of $G$, and two $\gamma$-sets are adjacent in $G(\gamma)$ if they differ by a single vertex and the two different vertices are adjacent in $G$.  In this paper, we consider $\gamma$-graphs of trees.  We develop an algorithm for determining the $\gamma$-graph of a tree, characterize which trees are $\gamma$-graphs of trees, and further comment on the structure of $\gamma$-graphs of trees and its connections with Cartesian product graphs, the set of graphs which can be obtained from the Cartesian product of graphs of order at least two. 
\end{abstract}

\section{Introduction}

Let $G$ be a graph and let $v$ be a vertex of $G$.  The \emph{open neighbourhood} of a vertex $v$, denoted $N(v)$, is the set of vertices adjacent to $v$, and the \emph{closed neighbourhood} of a vertex $v$, denoted $N[v]$, is $N(v) \cup \{v\}$.  For a subset of vertices $S$, we say $N(S) = \cup_{x \in S} N(x)$ and $N[S] = S \cup N(S)$.  A subset of vertices $D$ is a \emph{dominating set} of $G$ if $N[D] = V(G)$, that is, every vertex not in $D$ is adjacent to a vertex in $D$.  The \emph{domination number} of a graph $G$, denoted $\gamma(G)$, is the minimum cardinality of a dominating set of $G$.  A dominating set of minimum cardinality is said to be a \emph{$\gamma$-set}.

For a dominating set $D$ of $G$ and a vertex $x \in D$, the set of \emph{$D$-private neighbours} of $x$, denoted $pn(x, D)$, is $N[x] - N[D - x]$, that is, the set of vertices in the closed neighbourhood of $x$ and not in the closed neighbourhood of any other vertex in $D$.  If $x \in pn(x, D)$, then $x$ is a \emph{$D$-self private neighbour} in $D$ and if $y \neq x$ and $y \in pn(x, D)$, then $y$ is an \emph{$D$-external private neighbour} of $x$ in $D$.  If $D$ is a $\gamma$-set, then every vertex in $D$ has a private neighbour.

The \emph{$\gamma$-graph} of a graph $G$, introduced by Fricke et al.~\cite{FHHH11}, is a graph denoted $G(\gamma)$, has its vertex set as the $\gamma$-sets of $G$, and two $\gamma$-sets $D_1$ and $D_2$ are adjacent in $G(\gamma)$ if there are vertices $u \in D_1$ and $v \in D_2$ such that $D_2 = (D_1 - \{u\})\cup \{v\}$ and $uv \in E(G)$.  Starting with $D_1$, we think of making a \emph{swap}, that is, changing vertex $u$ for $v$, to form $D_2$.  A slightly different model, in which $uv$ need not be an edge in $G$, was introduced independently by Subramanian and Sridharan~\cite{SS08}; we do not consider this model here.  Fricke et al.~\cite{FHHH11} studied properties of $\gamma$-graphs, and raised the following open questions:
\begin{enumerate}
\item Is $\Delta(T(\gamma)) = O(n)$ for every tree $T$ of order $n$?
\item Is $\diam(T(\gamma)) = O(n)$ for every tree $T$ of order $n$?
\item Is $|V(T(\gamma))| \le 2^{\gamma(T)}$ for every tree $T$?
\item Which graphs are $\gamma$-graphs of trees?
\item Which graphs are $\gamma$-graphs?  Can you construct a graph $H$ that is not a $\gamma$-graph of any graph $G$?
\item For which graphs $G$ is $G(\gamma) \cong G$?
\item Under what conditions is $G(\gamma)$ a disconnected graph?
\end{enumerate}
The first three questions were solved by Edwards, MacGillivray, and Nasserasr~\cite{E15} and the fifth question was answered by Connelly, Hutson, and Hedetneimi~\cite{CHH11}; the other three questions remain open.  The question of which graphs are $\gamma$-graphs of trees was restated in a recent survey by Mynhardt and Nasserasr~\cite{MN} as a primary direction of study in this area.

Some $\gamma$-graphs of trees were determined by Fricke et al.~\cite{FHHH11}, as well as a couple of general properties.  A \emph{stepgrid} $SG(k)$ is the induced subgraph of the $k \times k$ grid graph $P_k \bbox P_k$ defined as follows: $SG(k) = (V(k), E(k))$, where 
\begin{align*}
V(k) &= \{(i, j): 1 \le i, j \le k, \ i + j \le k + 2\}, \\ 
E(k) &= \{((i, j), (i', j')): i' = i, \ j' = j + 1; \ i' = i + 1, \ i' = j\}.
\end{align*}

\begin{prop} \cite{FHHH11}
\begin{itemize}
\item $K_{1, n}(\gamma) \cong K_1$.
\item For $k \ge 3$, $K_{2, n}(\gamma) \cong K_{1, 2n}$.
\item $P_{3k}(\gamma) \cong K_1$.
\item $P_{3k + 2}(\gamma) \cong P_{k + 2}$.
\item $P_{3k + 1}(\gamma) \cong SG(k + 1)$.  
\end{itemize}
\end{prop}

\begin{theorem} \cite{FHHH11} \label{thm:tree-wo}
Let $T$ be a tree, and let $x \in V(T)$ be a vertex that does not appear in any $\gamma$-set of $T$.  Let $T_1, T_2, \ldots, T_k$ be the disjoint subtrees created by deleting $x$ from $T$, and let $x_i \in T_i$ be the vertex in subtree $T_i$ adjacent to the vertex $x$.  Let $D_i$ be the set of minimum dominating sets of subtree $T_i$ and let $T_i^{x_i}(\gamma)$ be the $\gamma$-graph of subtree $T_i$ using only those $\gamma$-sets of $D_i$ that do not contain $x_i$.  Then 
\[
T(\gamma) = T_1(\gamma) \bbox T_2(\gamma) \bbox \cdots \bbox T_k(\gamma) - (T_1^{x_1}(\gamma) \bbox T_2^{x_2}(\gamma) \bbox \cdots \bbox T_k^{x_k}(\gamma)).
\]
\end{theorem}

\begin{theorem} \cite{FHHH11} \label{thm:gt:con}
The $\gamma$-graph $T(\gamma)$ of every tree $T$ is a connected graph.
\end{theorem}

\begin{theorem} \cite{FHHH11}\label{bipartite}
For any tree $T$, $T(\gamma)$ is $C_n$-free, for any odd $n \ge 3$.  [That is, $T(\gamma)$ is bipartite.]
\end{theorem}

The main objective of this paper is the investigation of open question 4.  In Section~\ref{sec:alg}, we develop a method for constructing the $\gamma$-graph of a tree.  In Section~\ref{sec:gamma-tree}, we characterize those trees which are $\gamma$-graphs of trees.  Finally, in Section~\ref{sec:tree-gamma}, we state some further results for graphs which are $\gamma$-graphs of trees.

The following observations are used, frequently without reference, throughout the paper.

\begin{observation}\label{basic}
If $D_1$ and $D_2=(D_1- \{u\}) \cup \{v\}$ with $u\sim v$ are two distinct $\gamma$-sets of a graph, then $pn(D_1, u)=pn (D_2, v)\subseteq \{u,v\}$.
\end{observation}

\begin{observation}\label{spn}
If $u\in D$ is such that $pn(D,u)=\{u\}$, then for each $v\in N(u)$, we have that $D$ is adjacent to $D_v=(D- \{u\}) \cup \{v\}$ .
\end{observation}

\begin{observation}\label{spn2}
If $D_1$, $D_2=(D_1- \{u\}) \cup \{v\}$ and $D_3=(D_1- \{u\}) \cup \{w\}$ with $ v,w\in N(u)$, then $pn(D_1, u)=pn (D_2, v)= pn (D_2, w)=\{u\}$.
\end{observation}

\section{Computing $\gamma$-Graphs of Trees}
\label{sec:alg}

While Fricke et al.~\cite{FHHH11} developed a method to determine the $\gamma$-graphs of trees that contain a vertex not appearing in any $\gamma$-set (Theorem~\ref{thm:tree-wo}), there are trees for which every vertex appears in some $\gamma$-set.  We present a general algorithm to determine the $\gamma$-graph of a tree.  We first state the following results of Edwards, MacGillivray, and Nasserasr~\cite{E15} which will aid in verifying the algorithm.  If $D$ is a $\gamma$-set of a rooted tree $(T, c)$, then the \emph{height} of $D$, denoted $ht_T(D)$, is $\sum_{x \in D} d(x, c)$. Define $D$ to be  a \emph{higher} $\gamma$-set than $F$ if $ht_T(D) < ht_T(F)$ and $D$ to be a \emph{highest $\gamma$-set} if $ht_T(D) \le ht_T(F)$ for all $\gamma$-sets $F$ of $T$.

\begin{Lemma} \cite{E15} \label{prop:tree:high:pn}
A $\gamma$-set $D$ is a highest $\gamma$-set of a tree $T$ rooted at a vertex $c$ if and only if every $x \in D - \{c\}$ has a child $y \in pn(x, D)$.
\end{Lemma}

\begin{theorem} \cite{E15}\label{highest}
Let $T$ be a tree rooted at a vertex $c$.  Then $T$ has a unique highest $\gamma$-set.
\end{theorem}

The algorithm DOMSET, developed by Cockayne, Goodman, and Hedetniemi~\cite{CGH75} finds a minimum dominating set of a tree in linear time.  We demonstrate that with a slight modification to the algorithm not affecting the run time, i.e. we exclude the root from being an eligible end vertex, the algorithm finds the highest minimum dominating set of a rooted tree.

\begin{algorithm} 
Let $T$ be a tree rooted at a vertex $c$.  We construct the highest $\gamma$-set $S$ of $(T, c)$ as follows:
\begin{algorithmic}[1]
\Procedure{HIGHEST}{$T, c$}
	\State Set $S \gets \emptyset$; $G \gets T$; label each vertex of $T$ as {\em bound}.
	\While{$G$ has an endvertex $v \neq c$ adjacent to a vertex $u$}
		\If{$v$ is {\em free}}
			\State $G \gets G - v$.
		\ElsIf{$v$ is {\em bound}}
			\State Relabel $u$ as {\em required};
			\State $G \gets G - v$.
		\ElsIf{$v$ is {\em required}}
			\State $S \gets S \cup \{v\}$;
			\State If $u$ is  {\em bound} then relabel $u$ as {\em free};
			\State $G \gets G - v$.
		\EndIf
	\EndWhile
	\If{$c$ is not {\em free}}
		\State $S \gets S \cup \{c\}$.
	\EndIf 
\EndProcedure
\end{algorithmic}
\end{algorithm}

\begin{theorem}
Let $T$ be a tree rooted at a vertex $c$.  The output, $S$, of the algorithm HIGHEST$(T, c)$ is the highest $\gamma$-set of $T$.
\end{theorem}

\begin{proof}
Let $D$ be the $\gamma$-set output by HIGHEST$(T, c)$. By Lemma~\ref{prop:tree:high:pn}, $D$ is a highest $\gamma$-set of $T$ if and only if every $x \in D - \{c\}$ has a child $y \in pn(x, D)$.   Suppose there exists an $x \in D - \{c\}$ such that $x$ has no child in $pn(x, D)$. Let $y$ be a child of $x$. Either $y \in S$ or some child of $y$ was in $S$. In the first case, $y$ was relabelled {\em required} when it was processed by HIGHEST$(T, c)$ and in the latter case $y$ was labelled {\em required} or {\em free}, when it was processed by HIGHEST$(T, c)$. In either case $y$ was not labelled {\em bound}.  But then, as $y$ was arbitrary, $u$ could not have been labelled as {\em required} before being processed by HIGHEST$(T, c)$, and so $u \notin S$, contradicting our assumption.  Hence, together with Theorem~\ref{highest}, $S$ is the highest $\gamma$-set of $T$.
\end{proof}

We now present the following algorithm for determining the $\gamma$-graph of a tree. For each vertex $v$ in $T$, let $i(v)$ be its index in a breadth-first traversal of $T$.  We note that any such algorithm must be exponential as the number of gamma sets of a tree is potentially exponential.  Recently, Rote~\cite{Rote} provided a family of trees, referred to as the \emph{star of snowflakes}, with $13k + 1$ vertices and at least $95^k$ minimum dominating sets, i.e.\ trees whose number of $\gamma$-sets is on the order of $1.4194^n$, establishing a lower bound on the maximum number of minimum dominating sets. 

\begin{algorithm} \label{alg:gt}
Let $T$ be a tree rooted at a vertex $c$ and let $S$ be the highest $\gamma$-set of $T$.  For each vertex $v$ in $T$, let $i(v)$ be its index in a breadth-first traversal of $T$.  We construct the $\gamma$-graph of $T$, $T(\gamma) = (V, E)$, as follows:
\begin{algorithmic}[1]
\Procedure{GAMMATREE}{$T, c, S$}
	\State $V \gets \{S\}$; $E \gets \emptyset$.
	\State $\parents(S) \gets \emptyset$; $\children(S) \gets \emptyset$; $i(S) \gets 0$.
	\ForAll{$D \in V$}
		\ForAll{$v \in D$}
			\If{$pn(v, D) = \{v\}$}
				\State $\Swap \gets \{x : x \in N(v), i(x) > i (D)\}$.
			\ElsIf{$pn(v, D) - \{v\} = \{x\} \And i(x) > i(D)$}
				\State $\Swap \gets \{x\}$.
			\Else
				\State $\Swap \gets \emptyset$.
			\EndIf
			\ForAll{$x \in \Swap$}
				\State $D' \gets D - \{v\} \cup \{x\}$. \label{new-gamma}
				\State $V \gets V \cup \{D'\}$; $E \gets E \cup \{(D, D')\}$.
				\State $\parents(D') \gets \{(D, x)\}$; $\children(D') \gets \emptyset$; $i(D') \gets i(x)$. \label{alg:gt:index}
				\State $\children(D) \gets \children(D) \cup \{(D', x)\}$.
				\ForAll{$(A, a) \in \parents(D)$}
					\ForAll{$(B, b) \in \children(A)$}
						\If{$b = x$}
							\State $E \gets E \cup \{(D', B)\}$. \label{alg:gt:cycle}
							\State $\children(B) \gets \children(B) \cup \{(D', a)\}$.
							\State $\parents(D') \gets \parents(D') \cup \{(B, a)\}$.
						\EndIf
					\EndFor
				\EndFor
			\EndFor
		\EndFor
	\EndFor
\EndProcedure
\end{algorithmic}
\end{algorithm}

The remainder of this section is devoted to verifying the correctness of the algorithm GAMMATREE($T, c, S$).  The \emph{index} of a $\gamma$-set $D$, denoted here $i(D)$, is defined on line~\ref{alg:gt:index} of Algorithm~\ref{alg:gt}.  We first prove the following lemma regarding the value of the index of a $\gamma$-set.

\begin{Lemma} \label{lem:tree:index}
Let $T$ be a tree rooted at a vertex $c$ and let $S$ be the highest $\gamma$-set of $T$.  For every $\gamma$-set $D \neq S$ found by GAMMATREE($T, c, S$), $$i(D)=\max_{v\in D-S} i(v).$$
\end{Lemma}

\begin{proof}
Suppose $D'$ is the first $\gamma$-set of $T$ found by GAMMATREE($T, c, S$) such that $i(D')$ is not the largest index of the vertices of $D'$ not appearing in $S$. By definition of $i(D')$, there exists a $\gamma$-set $D$  and vertices $v, x$ such that $i(D') = i(x)$,  $D' = (D - \{v\}) \cup \{x\}$  and $i(x) > i(D)$. 

Suppose first that $y$ is a vertex in $D'-S$ such that $i(y) > i(D')$.   By choice of $D'$, we have that $i(D)=\max_{u\in D-S} i(u)$ and since $y \in D$, $i(D) \ge i(y)$.  Hence, $i(D') =i(x)> i(D)\ge i(y)> i(D')$, which is a contradiction. Therefore, we have $i(y) \le i(x)$ for all $y\in D'-S$ and if $x\in D'-S$ so we assume $x \in D' \cap S$. 

Hence any vertex in $D'$ with an index at least $i(x)$ must also be in $S$.   In particular, $x$, and every descendant of $x$ in $D'$ must also be in $S$.  By Lemma~\ref{prop:tree:high:pn}, $x$, and every descendant of $x$ in $D'$ has a child $S$-private neighbour which is also  a child $D'$-private neighbour. Let $w$ be a child $S$-private neighbour of $x$ which is also  a child $D'$-private neighbour of $x$. Clearly $w\notin S$. As $D$ is dominating and $D=(D' - \{x\}) \cup \{v\}$, it must be the case that $v=w$.  Hence, we obtain that $i(D) =\max_{u\in D-S} \ge i(w) > i(x)$, which is a contradiction as $i(x) > i(D)$.  The result follows.
\end{proof}

\begin{theorem} \label{thm:tree:vertex}
Let $T$ be a tree rooted at a vertex $c$ and let $S$ be the highest $\gamma$-set of $T$.  Every $\gamma$-set $D \neq S$ of $T$ is obtained by GAMMATREE($T, c, S$).
\end{theorem}

\begin{proof}
Assume there is a $\gamma$-set of $T$ not obtained by GAMMATREE($T, c, S)$.  For a $\gamma$-set $D = \{v_1, v_2, \ldots, v_{\gamma(T)}\}$ of $T$, let $I(D) = \{i(v_1), i(v_2), \ldots, i(v_{\gamma(T)})\}$, where $i(v_1) < i(v_2) < \cdots < i(v_{\gamma(T)})$.  Let $D'$ be the $\gamma$-set not obtained by GAMMATREE($T, c, S$) for which $I(D')$ is lexicographically smallest.  Let $x$ be so that $i(x)=i(D')$.  Then every descendant of $x$ in $D'$ has a higher index and hence is also in $S$.  It follows from Lemma~\ref{prop:tree:high:pn}, each descendant of $x$ in $D'$ has a child $D'$-private neighbour.  If $x$ also has a child $D'$-private neighbour, then every higher $\gamma$-set of $T$ must contain $x$, contradicting that $x \notin S$.  Hence, no child of $x$ is a $D'$-private neighbour.  Therefore, if $v$ is the parent of $x$, then the set $D = D' - \{x\} \cup \{v\}$ is a $\gamma$-set of $T$.  Since $I(D)$ is lexicographically smaller than $I(D')$, $D$ is found by GAMMATREE($T, c, S$).  Clearly each vertex in $D-S\subseteq (D'-S)\cup\{v\}$  has a smaller index than $x$, so $i(x) > i(D)$.  Further, as $D$ and $D'$ are both $\gamma$-sets of $T$, $pn(v, D) \subseteq \{v, x\}$, so for $v \in D$, $x \in \Swap$.  Hence, $D'$ is found by GAMMATREE($T, c, S$), which is a contradiction proving the theorem.
\end{proof}

Finally, we will show GAMMATREE($T, c, S$) obtains every edge of $T(\gamma)$.  We first demonstrate that the $\gamma$-sets are obtained by GAMMATREE$(T, c, S)$ in order of height.

\begin{Lemma} \label{lem:order-height}
The algorithm GAMMATREE$(T, c, S)$ can be implemented so that the $\gamma$-sets of $T$ are obtained in order of height.
\end{Lemma}

\begin{proof}
Let $D$ be a $\gamma$-set obtained by GAMMATREE($T, c, S)$.  As the algorithm runs, store the dominating sets of $T$ (these are the elements of $V$) in a queue, so that first one discovered by GAMMATREE($T, c, S)$ is the first one processed by GAMMATREE($T, c, S)$.  It suffices to  prove that every $\gamma$-set $D'$ found from $D$ (i.e.\ at line~\ref{new-gamma}) is such that $ht_T(D') = ht_T(D) + 1$.  We observe that there are adjacent vertices $v$ and $x$ such that $D' = (D - \{v\}) \cup \{x\}$.  If $D = S$, then every $z \in D - \{c\}$ has a child $y \in pn(z, D)$, so $x$ is a child of $v$ and $ht_T(D') = ht_T(D) + 1$ as desired.  Now suppose $D \neq S$.  If $i(v) > i(D)$, then by Lemma~\ref{lem:tree:index}, $v \in S$ and every descendant of $v$ which is in $D$ is also in $S$.  It follows that $v$ has a child $y \in pn(v, D)$, so $x$ must be a child of $v$, and $ht_T(D') = ht_T(D) + 1$ as desired.  Otherwise, $i(v) \le i(D)$ and $i(x) > i(D)$.  Then it is clear that $x$ is a child of $v$, and the result again follows.
\end{proof}

\begin{theorem} \label{thm:tree:edge}
Let $T$ be a tree rooted at a vertex $c$ and let $S$ be the highest $\gamma$-set of $T$.  The algorithm GAMMATREE$(T, c, S)$ can be implemented so that each edge in $T(\gamma)$ is obtained by GAMMATREE($T, c, S$).
\end{theorem}

\begin{proof}
Implement algorithm GAMMATREE$(T, c, S)$ as required in Lemma~\ref{lem:order-height} and suppose not all edges of $\gamma (T)$ are obtained.  Let $D_1$ be the first $\gamma$-set found by GAMMATREE($T, c, S$) such that there exists a $\gamma$-set $D_2$, found before $D_1$, where $D_2 = (D_1 - \{z\}) \cup \{y\}$, $yz \in E(T)$, but edge $D_1 D_2$ is not obtained by GAMMATREE($T, c, S$).  By Lemma~\ref{lem:order-height}, $ht(D_2) = ht(D_1)-1$.  As $D_1$ and $D_2$ are both $\gamma$-sets of $T$, we have $pn(z, D_1)=pn(y, D_2) \subseteq \{y, z\}$.  Suppose $i(x) = i(D_1)$.  Then  by Lemma~\ref{lem:order-height} $D_1$ was found while processing a set $D_3 = (D_1 - \{x\}) \cup \{w\}$, where $w$ is the parent of $x$ in $T$.  Clearly, neither $D_2$ nor $D_3$ is $S$.  Hence, there exists a $\gamma$-set $D_4$ such that $D_4 = (D_3 - \{z\}) \cup \{y\} = (D_2 - \{x\}) \cup \{w\}$.  Moreover, since $w$ is the parent of $x$, we have 
$$ ht_T(D_4) + 1 = ht_T(D_3) = ht_T(D_2) = ht_T(D_1) - 1, $$
so $D_4$ is found before $D_2$ and $D_3$ by Lemma~\ref{lem:order-height}.  Hence, $D_2 D_4$ and $D_3 D_4$ were both obtained by GAMMATREE($T, c, S$).  In particular, $(D_4, z) \in \parents(D_3)$ and $(D_2, x) \in \children(D_4)$.  Thus, we obtain the edge $D_1 D_2$ when processing $D_3$, which is the contradiction proving the theorem.
\end{proof}

The correctness of the algorithm immediately follows.

\begin{Corollary}
Let $T$ be a tree rooted at a vertex $c$ and let $S$ be the highest $\gamma$-set of $T$.  The algorithm GAMMATREE($T, c, S$) can be implemented to produce $T(\gamma)$.
\end{Corollary}

\begin{proof}
By Theorem~\ref{thm:tree:vertex}, we obtain every vertex of $T(\gamma)$, and by Theorem~\ref{thm:tree:edge}, we obtain every edge of $T(\gamma)$.  The result follows.
\end{proof}


\section{$\gamma$-Trees of Trees}
\label{sec:gamma-tree}

In this section, we characterize the trees which are $\gamma$-graphs of trees. We first consider properties of the $\gamma$-sets corresponding to the leaves of a $\gamma$-graph.

\begin{Lemma} \label{lem:gt:pn}
Let $T$ be a tree with at least three vertices and let $T(\gamma)$ be the $\gamma$-graph of $T$.  If $D$ is a leaf in $T(\gamma)$, then exactly one vertex  $v\in D$ has fewer than two $D$-external private neighbours. Furthermore either $v$ has exactly one $D$-external private neighbour or $v$ is both a leaf in $T$ and a $D$-self private neighbour.
\end{Lemma}

\begin{proof}
Note that if every vertex of dominating set $D$ in a tree has at least two $D$-external private neighbours, $D$ has no neighbours  in $T(\gamma)$.  If a vertex $x$ in $T$ has exactly one external private neighbour $y$ in $D$, then $(D - \{x\}) \cup \{y\}$ is a $\gamma$-set of $T$ adjacent to $D$.  If a vertex $z$ in $T$ has no external private neighbours in $D$, then $z$ must be a self private neighbour in $D$ and for every $w \in N(z)$, $(D - \{z\})\cup \{w\}$ is a $\gamma$-set of $T$ adjacent to $D$.  As $D$ is adjacent to exactly one $\gamma$-set, the result follows.
\end{proof}

Next, we prove a fundamental result on the number of private neighbours of a vertex in adjacent $\gamma$-sets of a tree.

\begin{Lemma} \label{lem:gt:change}
Let $D$ and $F$ be $\gamma$-sets of a tree $T$ which are adjacent in $\gamma (T)$ and let $x \in D \cap F$.  Then $|pn(x, D) \setminus pn(x, F) | \le 1$ and $|pn(x, F) \setminus  pn(x, D) | \le 1$.
\end{Lemma}

\begin{proof}
Suppose $F = (D - \{y\}) \cup \{z\}$ for adjacent vertices $y$ and $z$ in $T$.  Let $x \in D \cap F$. Suppose the result is false and assume without loss of generality that $| pn(x, D) \setminus   pn(x, F) | \ge 2$. 
If $x$ is a $D$-self private neighbour, but not an $F$-self private neighbour, then $z\sim x$ and $z$ is adjacent to a $D$-external private neighbour of $x$, say $w$.  Then $x w z$ is a 3-cycle in $T$, which contradicts that $T$ is a tree.  Otherwise, $x$ has at least two $D$-external private neighbours $v, w$ that are adjacent to  $z$, so $x v z w$ is a 4-cycle in $T$, which contradicts that $T$ is a tree.  The result follows.
\end{proof}

We now consider the $\gamma$-sets of leaves adjacent to the same stem.

\begin{Lemma} \label{lem:gt:stem}
Let $T$ be a tree and let $T(\gamma)$ be the $\gamma$-graph of $T$.  If $S$ is a stem in $T(\gamma)$ with degree at least three and $L_1, L_2, \ldots, L_k$ are the leaves of $S$, then there exists  vertices $x\in S$ and $y_1, y_2, \ldots, y_k\in V(T)-S$ so that $x\sim y_i$ and $L_i = (S - \{x\}) \cup \{y_i\}$ for $i=1, 2, \ldots k$.
\end{Lemma}

\begin{proof}
If $k = 1$, the statement is trivial, so assume $k \ge 2$.  As $S$ is adjacent to $L_i$ in $T(\gamma)$, for each $i$ there exists vertices $x_i$ and $y_i$ so that $L_i = (S - \{x_i\}) \cup \{y_i\}$ and $x_i\sim y_i$. Let $X = \{x_i: i = 1, 2, \ldots, k\}\subseteq S$.  Suppose $|X| \ge 2$.  By Lemma~\ref{lem:gt:pn}, $y_i$ is the only vertex in $L_i$ with fewer than two $L_i$-external private neighbours.  In particular, if $x_i\ne x_j$, then $x_i$ has at least two $L_j$-external private neighbours. It now follows from Lemma~\ref{lem:gt:change}, that each $x_i$ has an $S$-external private neighbour.  

If $x_i \neq x_j$, then $S - \{x_i, x_j\} \cup \{y_i, y_j\}$ is not a $\gamma$-set of $T$ as $L_i$ and $L_j$ are leaves of $T(\gamma)$.  Hence there exists a vertex $z_{ij}$ which is not dominated by $S - \{x_i, x_j\} \cup \{y_i, y_j\}$.  It must be the case that both $x_i$ and $x_j$ are neighbours of $z_{ij}$.  If $|X| \ge 3$, then each such $z_{ij}$ must be identical or $T$ is not a tree, but then $S - \{x_i, x_j\} \cup \{y_i, y_j\}$ is dominating, a contradiction.  Hence we assume $|X| = 2$, say $X=\{x_1, x_2\}$. Note that for $i=1,2$, $pn(S, x_i)=pn(L_i, y_i)\subseteq \{x_i, y_i\}$. As   each $x_i$ has an $S$-external private neighbour, $y_i$ is an $S$-private neighbour of $x_i$.  Hence for $i=1,2$, $x_i$ may only swap with $y_i$.  Furthermore as $y_i$ is the only vertex in $L_i$ with fewer than two $L_i$-external private neighbours, $x_1$ and $x_2$ have only one common neighbour, and $z_{12}$ and $T$ has no cycles, it follows that $x_1$ and $x_2$ are the only vertices in $S$ with fewer than two $S$-external private neighbours.  Thus, $S$, $L_i$, and $L_j$ are the only $\gamma$-sets of $T$, contradicting that $S$ has degree at least three.  Hence, $|X| = 1$, which completes the proof.
\end{proof}

Our next step is to show the graph $H$, pictured in Figure~\ref{fig:H}, is not a $\gamma$-graph of a tree.

\begin{figure}[htbp]
\begin{center}
\begin{tikzpicture}
\GraphInit[vstyle=Simple]
\grEmptyPath[prefix=a, RA = 1, rotation=90]{3}
\begin{scope}[shift={(1,0)}]
\grEmptyPath[prefix=b, RA = 1, rotation=90]{3}
\end{scope}
\EdgeInGraphSeq{a}{0}{1}
\EdgeInGraphSeq{b}{0}{1}
\Edge(a1)(b1)
\node at (-.5,1) {$S$};
\node at (-.5,0) {$L_1$};
\node at (-.5,2) {$L_2$};
\node at (1.5,1) {$R$};
\node at (1.5,0) {$M_1$};
\node at (1.5,2) {$M_2$};

\end{tikzpicture}
\caption{\label{fig:H} The graph $H$}
\end{center}
\end{figure}
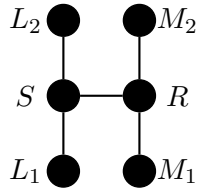

\begin{Lemma} \label{lem:gt:H}
No tree has $H$ as its $\gamma$-graph.
\end{Lemma}

\begin{proof}
Let $H$ be labeled as in Figure~\ref{fig:H} and suppose $H$ is the $\gamma$-graph of a tree $T$.  By Lemma~\ref{lem:gt:stem}, there exists vertices $x\in S$ and $w\in R$ so that $L_i = (S - \{x\}) \cup \{y_i\}$ and $M_i = (R - \{w\}) \cup \{z_i\}$ where $y_1, y_2\in N(x)$ and $z_1,z_2\in N(w)$.  Further there exists $u\in S$ so that $R = (S - \{u\}) \cup \{v\}$ with $u\sim v$. 
By Observation~\ref{spn2}, $x$ has no $S$-external private neighbours and $w$ has no $R$-external private neighbours. 

Suppose first $x \in R$.  Then $x\ne u$ and by Lemma~\ref{lem:gt:change}, $x$ has at most one $R$-external private neighbour.  Hence, there exists a vertex $y' \in N(x)$ such that $Q := (R - \{x\}) \cup \{y'\}$ is a $\gamma$-set of $T$. As $x\ne u$ and $R$ has only three neighbours, it follows that $Q\in \{M_1, M_2\}$ and therefore $x=w$. By Observation~\ref{spn}, we have $N(x)=\{y_1, y_2\}$ and $N(w)=\{z_1, z_2\}$ and therefore $\{y_1, y_2\}=\{z_1, z_2\}$.  Without loss of generality $y_1=z_1$.  Then $M_1 = L_1 - \{u\} \cup \{v\}$.  Thus, $L_1$ and $M_1$ are neighbours in $H$, which is a contradiction.  

Hence it must be that $x \notin R$.  By symmetry, $w \notin S$. It follows that $x=u$ and $w=v$.  In particular, $x\sim w$. But as $w$ has no $R$-external private neighbours, it is a self private neighbour in $R$, and hence must be an external private neighbour of $x$ in $S$.  But $x$ has no $S$-external private neighbours, which is a contradiction.  The result follows.
\end{proof}

We now show that if a graph is a $\gamma$-graph of a tree, then deleting leaves produces graphs that must also be $\gamma$-graphs of trees.  In particular, if a tree is a $\gamma$-graph of a tree, then all subtrees are also $\gamma$-graphs of trees.

\begin{Lemma} \label{unique}
If $G$ is a $\gamma$-graph of a tree $T$, $D$ is a leaf of $G$ and $x\in D$ has fewer than two $D$-external private neighbours, then $D$ is the unique $\gamma$-set of $T$ containing $x$.
\end{Lemma} 

\begin{proof}
 Consider $T'$ a component of $T-N[x]$.  Then $D\cap V(T')$ is clearly a dominating set of $T'$.  If $D\cap V(T')$ is not also a $\gamma$-set of $T'$, then $D$ is not minimal.   Hence $D\cap V(T')$ is a $\gamma$-set of $T'$ and since each vertex of $D-\{x\}$ has at least two $D$-external private neighbours, $D\cap V(T')$ has no neighbours in $T'(\gamma)$. It follows from Theorem~\ref{thm:gt:con} that $D\cap V(T')$ is the only $\gamma$-set of $T'$.  As $T$ is a tree, the vertices in $N(x)$ are adjacent to at most one vertex in $T'$.  Recall each vertex of $D\cap V(T')$ has at least two $D$-external private neighbours. Let $E$ be a $\gamma$-set of $T$ containing $x$. Then $E\cap V(T')$ must contain at least $|D\cap V(T')|$ vertices.  Since $T'$ was arbitrary,  $|D|=|E|=\gamma(T)$ and $|D\cap V(T')|=|E\cap V(T')|$, the set $E$ has at most one vertex in $N[x]$, namely $x$. For any component, $T'$ of $T-N[x]$,  $D\cap V(T')$ is the unique $\gamma$-set of $T'$.  Therefore $D$ is the unique $\gamma$-set of $T$ containing $x$.
\end{proof}

\begin{theorem} \label{thm:gt:leaf}
If $G$ is a $\gamma$-graph of a tree $T$ and $L$ is a leaf of $G$, then $G - L$ is a $\gamma$-graph of some tree $T'$.
\end{theorem}

\begin{proof}
By Lemma~\ref{lem:gt:pn}, exactly one vertex $x$ in $L$ has fewer than two $L$-external private neighbours.  Moreover, either $x$ has exactly one $L$-external private neighbour or $x$ is a leaf in $T$ and an $L$-self private neighbour. By Lemma~\ref{unique}, $L$ is the unique $\gamma$-set of $T$ containing $x$.

If  $x$ is a leaf of $T$ and is an $L$-self private neighbour, then $T'$ is formed by adding a leaf to the stem of $x$.
Suppose on the other hand, $x$ has exactly one external private neighbour $y$.  If $x$ is a self private neighbour, then $T'$ is formed by rooting $T$ at $y$, deleting the descendants of $x$, and adding a leaf to $y$.  Otherwise, $T'$ is formed by rooting $T$ at $y$, and deleting $x$ and its descendants.  \end{proof}

\begin{Corollary} \label{cor:gt:subtree}
If $G$ is a tree and a $\gamma$-graph of a tree $T$, then every subtree of $G$ is a $\gamma$-graph of some tree.
\end{Corollary}

\begin{Corollary} \label{cor:gt:H}
If $G$ is a tree and a $\gamma$-graph of a tree $T$, then $\mathbf{H}$ is not a subtree of $G$.
\end{Corollary}

\begin{proof}
The result immediately follows from Lemma~\ref{lem:gt:H} and Corollary~\ref{cor:gt:subtree}.
\end{proof}

From the previous corollary, we know that if a tree is a $\gamma$-graph of a tree, then the vertices of degree at least three are not adjacent.  We now wish to show  all trees without adjacent vertices of degree at least three are the $\gamma$-graph of a tree.  The proof is constructive.  We establish first some building blocks.

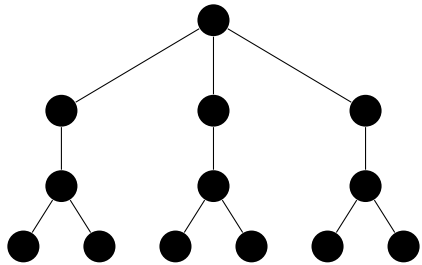
\begin{figure}[htbp]
\begin{center}
\begin{tikzpicture}
\node[vertex] (l1) at (0, 0) {};
\node[vertex] (l2) at (1, 0) {};
\node[vertex] (l3) at (2, 0) {};
\node[vertex] (l4) at (3, 0) {};
\node[vertex] (l5) at (4, 0) {};
\node[vertex] (l6) at (5, 0) {};

\node[vertex] (s1) at (0.5, 0.8) {};
\node[vertex] (s2) at (2.5, 0.8) {};
\node[vertex] (s3) at (4.5, 0.8) {};

\node[vertex] (t1) at (0.5, 1.8) {};
\node[vertex] (t2) at (2.5, 1.8) {};
\node[vertex] (t3) at (4.5, 1.8) {};

\node[vertex] (x) at (2.5, 3) {};

\draw (l1) -- (s1);
\draw (l2) -- (s1);
\draw (l3) -- (s2);
\draw (l4) -- (s2);
\draw (l5) -- (s3);
\draw (l6) -- (s3);
\draw (s1) -- (t1);
\draw (s2) -- (t2);
\draw (s3) -- (t3);
\draw (t1) -- (x);
\draw (t2) -- (x);
\draw (t3) -- (x);
\end{tikzpicture}
\caption{The graph $Y_3$}
\end{center}
\end{figure}

\begin{theorem} \label{thm:gt:star}
Let $Y_n$ be the graph obtained by taking $n$ copies of $K_{1, 3}$ and joining a leaf of each copy to a common vertex.  Then $Y_n(\gamma) \cong K_{1, n}$, and in every $\gamma$-set which is a leaf of $K_{1, n}$, every vertex has an external private neighbour.
\end{theorem}

\begin{proof}
Since every stem of $Y_n$ is adjacent to two leaves, every $\gamma$-set of $Y_n$ must contain every stem.  Only the common centre vertex is left to be dominated, so we can form a $\gamma$-set by adding this vertex or any of its neighbours.  Further, the only edges are between the $\gamma$-set containing the common centre vertex and each of the other $\gamma$-sets.  Finally, in each of these other $\gamma$-sets, every stem has at least two external private neighbours, its leaves, and the additional vertex has the common centre vertex as its private neighbour.  The result follows.
\end{proof}

We now provide a tool for combining these building blocks.

\begin{Lemma} \label{lem:gt:cut}
Suppose for $i=1,2$, $G_i$ is a $\gamma$-graph of tree $T_i$, $X_i$ is a leaf of $G_i$  and  every vertex in the dominating set $X_i$ has at least one $X_i$-external private neighbour. Then $G$, the graph formed by identifying $X_1$ and $X_2$ is the $\gamma$-graph of a tree $T :\cong T_1 \oplus^{v_1}_{v_2} T_2$, the tree formed by linking two $T_1$ and $T_2$ by creating a new vertex $v$ adjacent to $v_1\in V(T_1)$ and $v_2\in V(T_2)$.  

Furthermore let $Y\in V(G_1)$ be a dominating set of $T_1$ and $Y'\in V(G)$ be the corresponding dominating set of $T$.  If every vertex of $Y$ has a $Y$-external neighbour, then every vertex in $Y'$ has a $Y'$-external neighbour.
\end{Lemma}

\begin{proof}
By Lemma~\ref{lem:gt:pn}, exactly one vertex $v_i$ in $X_i$ has fewer than two external private neighbours.  It follows that $v_i$ has exactly one external private neighbour.  Form $T$ from $T_1$ and $T_2$ by adding a new vertex $v$ with $N(v)= \{v_1,v_2\}$.  We now show that $T(\gamma) \cong G$.

It is clear that $\gamma(T) \le \gamma(T_1) + \gamma(T_2)$.  Suppose some $\gamma$-set $D$ of $T$ contains $v$.  Then without loss of generality, $\gamma(T_1 - \{v_1\}) < \gamma(T_1)$.  Let $D'$ be a $\gamma$-set of $T - \{v_1\}$.  Then no neighbour of $v_1$ is in $D'$.  Hence $D' \cup \{v_1\}$ is a $\gamma$-set of $T_1$ and $v_1$ has no external private neighbours.  Then, by Lemma~\ref{unique}, $D' \cup \{v_1\}$ is the unique $\gamma$-set of $T_1$ containing $v_1$, so $D' \cup \{v_1\} = X_1$.  But this contradicts the presupposition that every vertex in $X_1$ has an $X_1$-external private neighbour.  Hence no $\gamma$-set of $T$ contains $v$, $\gamma(T) = \gamma(T_1) + \gamma(T_2)$ and every $\gamma$-set of $T$ is the union of a $\gamma$-set of $T_1$ and a $\gamma$-set of $T_2$.

Every $\gamma$-set of $T$ must contain $v_1$ or $v_2$ for $v$ to be dominated. Consider $X = X_1 \cup X_2$.  It is clear that $X$ is the unique $\gamma$-set containing both $v_1$ and $v_2$.  Since every other vertex in $X$ has least least two external private neighbours, and $v_1$ and $v_2$ have exactly one external private neighbour, $X$ is adjacent to two $\gamma$-sets. Let $D$ be a dominating set of $T$ other than $X$.  If $v_2\in D$ and by Lemma~\ref{unique}, $D\cap T_2\cong X_2$ and every vertex in $D\cap T_2$ has at least two $D$-external private neighbours.  Therefore the subgraph of the  $\gamma$-sets of $T$ containing $v_2$  in $T(\gamma)$ is $G_1$ and similarly the subgraph of the  $\gamma$-sets of $T$ containing $v_1$  in $T(\gamma)$ is $G_2$.  Hence, the $\gamma$-graph of $T$ is $G$ as required. The final statement follows directly from the above construction.
\end{proof}

We can now prove the main result.

\begin{theorem}\label{trees}
If $G$ is a tree, then $G$ is a $\gamma$-graph of some tree if and only if $H$ is not a subtree of $G$.
\end{theorem}

\begin{proof}
Necessity follows by Corollary~\ref{cor:gt:H}; it remains to show sufficiency.  The result trivially holds for $K_1$ as $K_1(\gamma) \cong K_1$; so suppose $G$ has at least two vertices.  We proceed by induction on $k$, the number of degree two vertices of $G$.  If $k = 0$, then $G \cong K_{1, n - 1}$, and by Theorem~\ref{thm:gt:star}, $K_{1, n - 1}$ is the $\gamma$-graph of $Y_{n - 1}$, and in every $\gamma$-set which is a leaf of $K_{1, n - 1}$, every vertex has an external private neighbour.

Assume $k > 0$ and that every tree $F$ not containing $H$ as a subtree with fewer than $k$ degree two vertices is the $\gamma$-graph of some tree and in every $\gamma$-set which is a leaf of $F$, every vertex has an $F$-external private neighbour.  Let $x$ be a degree two (cut-)vertex in $G$, and let $G_1$ and $G_2$ be the two components of $G - \{x\}$.  As $G_1 \cup \{x\}$ and $G_2 \cup \{x\}$ each have fewer than $k$ degree two vertices, then by the induction hypothesis, each is the $\gamma$-graph of a tree (say $T_1$ and $T_2$ respectively) and in every $\gamma$-set which is a leaf of $T_1$ and $T_2$ repsectively, every vertex has an external private neighbour.  By Lemma~\ref{lem:gt:cut}, there are vertices $v_1\in V(T_1)$ and $v_2\in V(T_2)$  so that the $\gamma$-graph of $T_1 \oplus^{v_1}_{v_2} T_2$ is $G$ and by construction, in every $\gamma$-set which is a leaf of $G$, every vertex has an external private neighbour.  The result follows.
\end{proof}


\section{Properties of $\gamma$-Graphs of Trees}
\label{sec:tree-gamma}

In this section, we investigate  general properties of $\gamma$-graphs of trees.  We begin with a result classifying the edges in $\gamma$-graphs of trees, highlighting the importance of cut edges and 4-cycles in these graphs.   

\begin{Lemma}\label{4cycle or cut edge}
Let $T$ be a tree and let $T(\gamma)$ be the $\gamma$-graph of $T$.  Every edge  of $T(\gamma)$ is a cut-edge or is contained in a 4-cycle.
\end{Lemma}

\begin{proof}
Let $e$ be an edge which is not a cut-edge of $T(\gamma)$.  Then $e$ is contained in a cycle.  But cycles in Algorithm~\ref{alg:gt} are only created at line~\ref{alg:gt:cycle}, which always creates a 4-cycle.  The result follows.
\end{proof}

 We now show that a cut-edge, not incident with a leaf allows you to decompose a $\gamma$-graph of a tree into two smaller graphs, both of which are also $\gamma$-graphs of trees. 

\begin{Lemma}\label{cutedge}
Let $G$ be a $\gamma$-graph of a tree $T$ with a cut-edge $e$, and let $G_1$ and $G_2$ be the components of $G \setminus e$.  Then $G_1 + e$ and $G_2 + e$ are each $\gamma$-graphs of trees.
\end{Lemma}

\begin{proof}
Let $e = AB$ with $A \in V(G_1)$, $B \in V(G_2)$ and $A = (B - \{y\}) \cup \{z\}$ for some vertices $y, z \in V(T)$.  Let $\{z\} \cup \{z_i\}$ be the set of vertices that can be swapped from $B$.  Form $T'$ by adding a leaf to each $z_i$.  Since $e$ is a cut-edge, $z_i$ cannot be swapped in $A$, and hence each $z_i$ has exactly one external private neighbour in $B$.  But $B$ is also a $\gamma$-set of $T'$, and each $z_i$ has two external private neighbours.  Moreover, each $\gamma$-set in $G_1$ is also a $\gamma$-set of $T'$, so it follows that $G_1 + e$ is the $\gamma$-graph of $T'$.  By symmetry, $G_2 + e$ is also the $\gamma$-graph of a tree.
\end{proof}

\begin{figure}[htbp]
\begin{center}
\begin{tikzpicture}
		\node [style=vertex] (0) at (-2, 2) {};
		\node [style=vertex] (1) at (-4, 2) {};
		\node [style=vertex] (2) at (-4, 0) {};
		\node [style=vertex] (3) at (-2, 0) {};
		\node [style=vertex] (4) at (-2, -2) {};
		\node [style=vertex] (5) at (-4, -2) {};
		\node [style=vertex] (6) at (0, 0) {};
		\node [style=vertex] (7) at (0, -2) {};
		\node [style=vertex] (8) at (2, 0) {};
		\node [style=vertex] (9) at (2, 2) {};
		\node [style=vertex] (10) at (4, 0) {};
		\node [style=vertex] (11) at (2, -2) {};
		\node  at (-1,.2) {$e$};
		\draw (1) to (2);
		\draw (1) to (0);
		\draw (0) to (3);
		\draw (2) to (3);
		\draw (2) to (5);
		\draw (5) to (4);
		\draw (4) to (3);
		\draw (3) to (6);
		\draw (6) to (7);
		\draw (6) to (8);
		\draw (8) to (9);
		\draw (8) to (11);
		\draw (8) to (10);
\end{tikzpicture}
\caption{The graph $Z$}\label{example}
\end{center}
\end{figure}
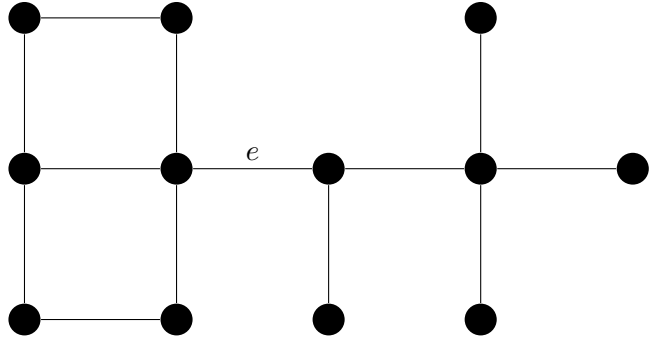

Setting $G=H$ (see Figure~\ref{fig:H}) and $e$ the edge joining the two stems of $H$, shows the converse of Lemma~\ref{cutedge} does not necessarily hold. That is both $G_1 + e=K_{1,3}$ and $G_2 + e=K_{1,3}$ are each $\gamma$-graphs of trees, but $G=H$ is not a $\gamma$-graph of any tree. Lemma~\ref{cutedge} does provide a tool for showing a graph with a cut-edge not incident to a leaf is not the $\gamma$-graph of any tree $T$.  For example, consider the graph $Z$ shown in Figure~\ref{example}.  Let $Z_1$ and $Z_2$ be the components of $Z$ obtained by deleting the edge $e$, where $Z_2$ is a tree.  Then $Z_2 +e$ is a tree with adjacent vertices of degree at least 3 and thus by Theorem~\ref{trees},  $Z_2 +e$ is not the  $\gamma$-graph of any tree $T$.  It now follows from Lemma~\ref{cutedge} that $Z$ is not the $\gamma$-graph of any tree $T$.

The remainder of the paper will focus on $\gamma$-graphs of trees where every edge is in a 4-cycle or incident with a leaf.  We first present the following useful result due to Edwards, MacGillivray, and Nasserasr~\cite{E15}, which will be used throughout.

\begin{Lemma} \cite{E15} \label{prop:gt:unique}
For a $\gamma$-set $D$ of a tree $T$ and a vertex $z \notin D$, there is at most one vertex $v \in D$ such that $(D - \{v\})\cup \{z\}$ is also a $\gamma$-set of $T$.
\end{Lemma}

\subsection{Cartesian Product and $\gamma$-Graphs of Trees}

Every edge in the Cartesian product of two connected graphs of order at least two is in a 4-cycle.  Theorem~\ref{thm:tree-wo} highlights one aspect of the connection between $\gamma$-graphs and the Cartesian product.  When movement of the dominating set in one part of a graph has no effect on movement of the dominating set in another part of the graph, the Cartesian product frequently arises. In this section we exploit this connection and establish that the  Cartesian product  can both decompose $\gamma$-graphs of trees into smaller $\gamma$-graphs of trees and be used to combine $\gamma$-graphs of trees to discover other $\gamma$-graphs of trees.  We first consider {\em Cartesian product graphs}, the set of graphs which can be obtained from the Cartesian product of graphs of order at least two.  Sabidussi~\cite{S60} and Vizing~\cite{V63} independently demonstrated a prime factorization of graphs, analogous to the Fundamental Theorem of Arithmetic.

\begin{theorem} \cite{S60, V63}
All finite connected graphs have a unique prime factorization with respect to Cartesian multiplication.
\end{theorem}

We show that if $G$ is a `composite' graph, then each of its prime factors are $\gamma$-graphs of trees if and only if $G$ is also the $\gamma$-graph of a tree.  Hence, to characterize $\gamma$-graphs of trees, we need only consider `prime' graphs under Cartesian multiplication.

\begin{theorem}
Let $G$ be a Cartesian product graph with $G = G_1 \bbox G_2$, $G_i \neq K_1$.  Then $G$ is the $\gamma$-graph of a tree $T$ if and only if each $G_i$ is the $\gamma$-graph of a tree $T_i$.
\end{theorem}

\begin{proof}
Suppose $G = G_1 \bbox G_2$, $G_1, G_2 \neq K_1$, is the $\gamma$-graph of a tree $T$.  Let $e = U_1 U_2$ be an edge of $G_1$ and let $e_1$ and $e_2$ be any two corresponding edges in $G$ on vertices $V_1$ and $V_2$ of $G_2$ respectively.  As $G$ is connected by Theorem~\ref{thm:gt:con}, there is a path between $(U_1, V_1)$ and $(U_1, V_2)$ in $G$, and considering only the edges in $G_2$ gives a path $P$ between $V_1$ and $V_2$.  Consider the subgraph $e \bbox P$ of $G$.  It follows from Proposition~\ref{prop:4cycle} that every edge corresponding to $e$ makes the same swap between its two $\gamma$-sets.  By symmetry, this is true of every edge in $G_1$ or $G_2$.  Let $(W, X)$ be a vertex of $G$.  If a vertex $a \in (W, X)$ is swapped by an edge of $G_2$, it is involved in no swap in $G_1$ by Proposition~\ref{prop:4cycle}.  Hence, form $T_1$ by adding two leaves to any vertex in $(W, X)$ that is swapped in $G_2$.  Then clearly, $T_1(\gamma) \cong G_1$.  By symmetry, $G_2$ is also the $\gamma$-graph of a tree.

Conversely, suppose each $G_i$ is the $\gamma$-graph of a tree $T_i$.  Let $T$ be the graph formed by adding an edge between any stem of $T_1$ and any stem of $T_{2}$.  It is easily verified that $T(\gamma) \cong G$.
\end{proof}

\subsection{Incident 4-cycles in $\gamma$-Graphs of Trees}

In this subsection we give some structure to how four cycles interact locally in $\gamma$-graphs of a trees. We first demonstrate that $K_{2, 3}$ cannot be a subgraph of a $\gamma$-graph of a tree.  As a consequence, we observe that two 4-cycles have at most two common vertices, and moreover, if two 4-cycles have two common vertices, then the two common vertices are adjacent.

\begin{theorem} \label{thm:noK23}
Let $T$ be a tree and let $T(\gamma)$ be the $\gamma$-graph of $T$.  Then $T(\gamma)$ does not contain $K_{2, 3}$ as a subgraph.
\end{theorem}

\begin{proof}
For $i\in\{1,2,3\}$ and $j\in\{1,2\}$ let $Y_i$ and $X_j$ be vertices of $T(\gamma)$ that induce $K_{2, 3}$.  That is $X_j\sim Y_i$ for each possible $i$ and $j$.  Then for each $i\in\{1,2,3\}$ there exist vertices $a_i\in X_1$ and $b_i\notin X_1$ with $a_i\sim b_i$ and $Y_i = (X_1 - \{a_i\}) \cup \{b_i\}$.  By Lemma~\ref{prop:gt:unique}, it must be the case that  $b_1$, $b_2$ and $b_3$ are all distinct.  It follows from Theorem~\ref{bipartite} $X_1$ and $X_2$ are at distance two in $T(\gamma)$.  Therefore $|X_2 \setminus X_1| \le 2$ and so for some $i$, $b_i\notin X_2 \setminus X_1$. Assume with out loss of generality $b_3\notin X_2 \setminus X_1$.  As $b_3\notin X_1$ it must be the case that $b_3 \notin X_2$.  Then for some vertex $c\notin Y_3$ with $c\sim b_3$, $X_2 = (Y_3 - \{b_3\}) \cup \{c\}$, and it follows that $X_2 = (X_1 - \{a_3\}) \cup \{c\}$ and therefore $a_3\ne c$.  As each $b_i$ is distinct, we can assume without loss of generality that $b_2 \neq c$.  Then $X_2 = (Y_2 - \{a_3, b_2\}) \cup \{a_2, c\}$.  But $d(X_2, Y_2) = 1$, so $|X_2 \setminus Y_2| = 1$, which shows it must be the case that $a_2=a_3$. Clearly $b_2$ and  $b_3$ are both adjacent to $a_3$ in $T$.  Furthermore $X_2$ is adjacent to $Y_2$ in $G$ and $X_2 = (Y_2 - \{b_2\}) \cup \{c\}$ so it must be the case that $b_2$ is adjacent to $c$ in $T$.  By definition, $c\sim b_3$ so either $ c b_2 a_3 b_3$ forms a cycle in $T$ or $c=a_3$.  Both cases lead to a contradiction.
\end{proof}

We now demonstrate a key property of the $\gamma$-sets corresponding to a 4-cycle of $T(\gamma$) that we will make extensive use of in our next set of results.

\begin{prop} \label{prop:4cycle}
If $G$ is a $\gamma$-graph of a tree $T$ and $WXYZ$ is a 4-cycle of $G$, then for some distinct $a, b\in W$ and distinct $c,d\notin W$ with $a\sim c$ and $b\sim d$, so that $X= (W - \{a\}) \cup \{c\}$,  $Z= (W - \{b\}) \cup \{d\}$ and $Y = (W - \{a, b\}) \cup \{c, d\}= (Z - \{a\}) \cup \{c\}= (X - \{b\}) \cup \{d\}$.
\end{prop}

\begin{proof}
For some $p,r\in W$ and $q,s\notin W$ with $p\sim q$ and $r\sim s$ we have that $X = (W - \{p\}) \cup \{q\}$ and $Z = (W - \{r\}) \cup \{s\}$.  It follows from Lemma~\ref{prop:gt:unique} that $q \neq s$.  If $q, s \notin Y$, then as $Y\ne W$, for some vertex $t\notin X\cup Z$, with $q,s\in N(t)$, $Y = (X - \{q\}) \cup \{t\} = (Z - \{s\}) \cup \{t\}$ and $p = r$ with $p\notin Y$.  But then $p q t s$ is a 4-cycle in $T$, which contradicts that $T$ is a tree. Hence either $q\in Y$ or $s\in Y$.  Suppose without loss of generality $q \in Y$.  In the case that $s \notin Y$, then $Y =  (Z - \{s\}) \cup \{q\}$ and therefore $(X - \{r\})\cup \{p\}$.  This implies $s$ is adjacent to $q$ in $T$ and $r$ is adjacent to $p$ in $T$ (with $p\ne r$).  But then $p q s r$ is a 4-cycle in $T$, contradicting that $T$ is a tree.  Hence, $q, s \in Y$.  If $p=r$, then there exists a vertex $u\in W$ 
so that $Y = (X - \{ u\} )\cup \{ s\}= (Z - \{ u\} )\cup \{q\}$ and hence both $s$ and $q$ are adjacent to $u$ in $T$.  It follows that $s u q p(=r)$ is a cycle in $T$, showing that $p$ and $r$ are distinct. Thus $Y = (W - \{p, r\} )\cup \{q, s\}$, as required. 
\end{proof}

We have shown that two 4-cycles in $T(\gamma)$ cannot overlap in more than two vertices.  On the other hand, we can also say something about the structure required in $T(\gamma)$ when two 4-cycles have only one common vertex.  We first establish the following.

\begin{prop} \label{prop:R:sets}
Let $G$ be a $\gamma$-graph of a tree $T$ and let $C = WXYZ$ be a 4-cycle of $G$. By Proposition~\ref{prop:4cycle} there are distinct $a, b\in W$ and distinct $c,d\notin W$ with $a\sim c$ and $b\sim d$, so that $X= (W - \{a\}) \cup \{c\}$,  $Z= (W - \{b\}) \cup \{d\}$ and $Y = (W - \{a, b\}) \cup \{c, d\}$.  If $U$ and $V$ are each adjacent to $W$ but not part of a 4-cycle with a pair of vertices of $C$ and $U = (W - \{e\}) \cup \{f\}$ and $V = (W - \{g\}) \cup \{h\}$, then $e= g$.
\end{prop}

\begin{proof}
Suppose the opposite.  It follows from Lemma~\ref{prop:gt:unique} that $c$, $d$, $f$ and $h$ are all distinct.  As neither $U$ nor $V$ is part of a 4-cycle with any two vertices of $C$,  $(W - \{a, e\}) \cup \{c, f\}$, is not a $\gamma$-set of $T$. Hence, if $a\ne e$, $a$ and $e$ have common neighbour $i$ which is not adjacent to any vertex in $(W - \{a, e\}) \cup \{c, f\}$.  As $b, g\in  (W - \{a, e\}) \cup \{c, f\}$, $i$ is not adjacent to $b$ or $g$ in $T$.  Similarly it can be seen $(W - \{b, e\}) \cup \{d, f\}$, $(W - \{a, g\}) \cup \{c, h\}$, and $(W - \{b, g\}) \cup \{d, h\}$ are not $\gamma$-sets of $T$. Hence, if $b\ne e$, $b$ and $e$ have common neighbour $j$ not adjacent to $a$ or $g$, if $a\ne g$, then $a$ and $g$ have common neighbour $k$ not adjacent to $b$ or $e$, and if $b\ne g$, $b$ and $g$ have common neighbour $l$ not adjacent to $a$ or $e$.  It can be seen that $i, j, k, l$ (if they exist) are all distinct.  

Suppose $a=e$.  Then as $a$ and $b$ are distinct, $b\ne e$. Hence, $b$ and $e$ have common neighbour $j$ not adjacent to $a=e$, a contradiction.  Hence $a\ne e$. Similarly we can show $b\ne e$, $a\ne g$ and $b\ne g$.  But then $a i e j b l g k$ is a cycle in $T$, contradicting that $T$ is a tree.  The result follows.
\end{proof}

\begin{prop} \label{lem:no-link}
Let $G$ be a $\gamma$-graph of a tree $T$ and let $WXYZ$ and $WVUS$ be 4-cycles of $G$.  Then there must be a 4-cycle $WRQP$, with $R \in \{X, Z\}$ and $P \in \{S, V\}$.
\end{prop}

\begin{proof}
Suppose not.  Then neither $S$ nor $V$ is part of a 4-cycle with a pair vertices of $WXYZ$.  By Proposition~\ref{prop:4cycle},  there are distinct vertices $a, b\in W$ and $c, d\notin W$,   with   $a\sim c$ and $b\sim d$  so that $S = (W - \{a\}) \cup \{c\}$ and $V = (W - \{b\}) \cup \{d\}$.  By Proposition~\ref{prop:R:sets}, $a=b$, a contradiction.  The result follows.
\end{proof}

Define two 4-cycles of $G$ to be {\em adjacent} if they share an edge and two adjacent cycles to be neighbours.  Then Proposition~\ref{lem:no-link} can be reworded to say that incident 4-cycles in a $\gamma$-graph of a tree must be adjacent to each other or have a common neighbour. We conclude our treatment of vertices of 4-cycles by demonstrating that adjacent vertices of a 4-cycle cannot both have vertex neighbours that are not part of neighbouring 4-cycles.

\begin{prop}\label{adjacent}
Let $G$ be the $\gamma$-graph of a tree and let $C$ be an induced $C_4$ of $G$.  Then in any two adjacent vertices of $C$ at most one has a neighbour which is not part of a 4-cycle with a pair of vertices of $C$.
\end{prop}

\begin{proof}
Suppose $G$ is the $\gamma$-graph of a tree $T$.  Let $C = WXYZ$ and suppose that $W$ and $X$ have neighbours $U$ and $V$ respectively, which are not part of a 4-cycle with any two vertices of $C$.  By Proposition~\ref{prop:4cycle}, there are distinct vertices $a, b\in W$ and $c,d\notin W$ so that $a\sim c$, $b\sim d$, $X = (W - \{a\}) \cup \{c\}$, $Z = (W - \{b\})\cup \{d\}$ and $Y = (W - \{a, b\}) \cup \{c, d\}$.  There are vertices $e\in W$, $g\in X$, $f\notin W$ and $h\notin X$ so that $e\sim f$, $g\sim h$, $U = (W - \{e\}) \cup \{f\}$ and $V = (X - \{g\}) \cup \{h\}$.  We note from Lemma~\ref{prop:gt:unique} that $d$, $f$  and $c$ are all distinct and that $d$, $h$ and $a$ are all distinct.
As neither $U$ nor $V$ is part of a 4-cycle with any two vertices of $C$ we get:
\begin{itemize}
\item if $a\ne e$, then as $(W - \{a, e\}) \cup \{c, f\}$ is not a $\gamma$-set of $T$,  $a$ and $e$ have a common neighbour $i\notin N[b]\cup N[c]$.
\item if $b\ne e$, then as $(W - \{b, e\}) \cup \{d, f\}$ is not a $\gamma$-set of $T$, $b$ and $e$ have a common neighbour $j\notin N[a]$.
\item if $b\ne g$, then as $(X - \{b, g\}) \cup \{d, h\}$  is not a $\gamma$-set of $T$, $b$ and $g$ have a common neighbour $k\notin N[c]$.
\item if $c\ne g$, then as $(X - \{c, g\}) \cup \{a, h\}$ is not a $\gamma$-set of $T$, $c$ and $g$ have a common neighbour $l\notin N[a]\cup N[b]$.
\end{itemize}  

 Note that in the case they exist, $i\notin \{j,k,l\}$ and $k\ne l$.  If $e\notin\{a,b\}$ and $g\notin\{b,c\}$, then $a c l g k b j e i$  is a cycle in $T$, contradicting that $T$ is a tree. Hence, by symmetry, we may assume either $e=a$ or $e=b$.

Assume first that $e = a$.  Then as $a$ and $b$ are distinct, $b\ne e$. Therefore, from above $a=e$ and $b$ have a common neighbour $j$.  Since $Y= (W - \{a, b\}) \cup \{c, d\}$ is a $\gamma$-set of $T$, it must be that $j=c$ (if $j\sim c$, then $a j c$ is a 3-cycle in $T$).  If $b\ne g$ and $c\ne g$, then as $k\ne c$,  $b c l g k$ is a cycle in $T$, which contradicts that $T$ is a tree.  If $g = b$, as $b\in W$ and $c\notin W$, $c\ne g.$ From above $c$ and $g=b$ have a common neighbour, but $c$ and $b$ are adjacent, contradicting that $T$ is a tree. If $c=g$,  then the path $U W X V$ in $G$ corresponds to a swap of $f$ and $a$, followed by a swap of $a$ and $c$, followed by a swap of $c$ and $h$.  This corresponds to swaps of the four distinct vertices which induce a path $f a c h$ in $T$ and hence $f$ and $h$ are distance 2 apart in $T$.  By Observation~\ref{basic}, $pn(f, U)=pn(a,W)=pn(c,X)=pn(h,V)$, but no vertex is in the closed neighbourhood of both $f$ and $h$, a contradiction.

Otherwise, assume $e = b$.  Then as $a$ and $b$ are distinct, $a\ne e$.   Therefore from above, $a$ and $b=e$ have a common neighbour, and since $Y$ is a $\gamma$-set, it must be $d$.  If $b\ne g$ and $c\ne g$, then either $a c k g l b d$ is a cycle in $T$ or contains a cycle in $T$, which contradicts that $T$ is a tree.  If $b=g$, then $c\ne g$. So from above $c$ and $g=b$ have a common neighbour $l$. As $l\notin N[b]$, $l\ne d$ so $a c l b d$ is a cycle in $T$, contradicting that $T$ is a tree.  If $g = c$, then $g\ne b$.  From above $b$ and $g=c$ have a common neighbour $k$.   As $k\notin N[c]$, $k\ne a$ so $a c k b d$ is a cycle in $T$, contradicting that $T$ is a tree.  The result follows.
\end{proof}

While the evidence presented in this section is not conclusive, it highlights that locally, 4-cycles interact similarly to how 4-cycles interact in Cartesian Product graphs.

\subsection{The Structure of Stems in 4-Cycles}

Finally, we explore some properties of stems in a 4-cycle in the $\gamma$-graph of a tree. In particular we focus on the structure of the dominating set associated with the stem and the corresponding structure of the $\gamma$-graph. We first obtain the following lemma.

\begin{Lemma}\label{k=1}
Let $G$ be a $\gamma$-graph of a tree $T$,  $S$ be a stem of $G$,  $L_1, L_2, \ldots, L_k$ be the leaves of $S$, and  $G' = G - \{L_1, L_2, \ldots, L_k\}$.  If every vertex in $S$ has an $S$-external private neighbour, then $k=1$.  
\end{Lemma}

\begin{proof}
By Lemma~\ref{lem:gt:stem}, 
there exists  vertices $x\in S$ and $y_1, y_2, \ldots, y_k\in V(T)-S$ so that $x\sim y_i$ and $L_i = (S - \{x\}) \cup \{y_i\}$ for $i=1, 2, \ldots k$. By Lemma~\ref{lem:gt:pn}, every vertex except $y_j$ in $L_j$ has at least two $L_j$-external private neighbours, so by Lemma~\ref{lem:gt:change}, every vertex other than $x$ in $S$ has at least one $S$-external private neighbour. It follows that $x$ has an $S$-external private neighbour and therefore by Observation~\ref{spn2}, $k = 1$. 
\end{proof}

\begin{prop} \label{prop:gt:cart-spn}
Let $G$ be a $\gamma$-graph of a tree $T$,  $S$ be a stem of $G$,  $L_1, L_2, \ldots,$ $L_k$ be the leaves of $S$, and  $G' = G - \{L_1, L_2, \ldots, L_k\}$.  If $\deg_{G'}(S) \ge 3$, then $G'$ is a Cartesian product graph if and only if  every vertex in $S$ has an $S$-external private neighbour.
\end{prop}

\begin{proof}
By Lemma~\ref{lem:gt:stem}, 
there exists  vertices $x\in S$ and $y_1, y_2, \ldots, y_k\in V(T)-S$ so that $x\sim y_i$ and $L_i = (S - \{x\}) \cup \{y_i\}$ for $i=1, 2, \ldots k$.
By Lemma~\ref{lem:gt:pn}, every vertex except $y_j$ in $L_j$ has at least two $L_j$-external private neighbours, so by Lemma~\ref{lem:gt:change}, every vertex other than $x$ in $S$ has at least one $S$-external private neighbour.  Let $\{x \} \cup Z$ be the set of vertices that can be swapped from $S$. If $z\in Z$, then $z$ has two $L_j$-external private neighbours, but has at exactly one $S$-external private neighbour. Therefore each $z\in Z$ has a distinct common neighbour with $x$ which is not adjacent to any vertex of $S-\{x,z\}$.  Furthermore, rooting $T$ at $x$, each $z\in Z$ is in a distinct branch. By Observation~\ref{basic}, each $z\in Z$ is associated with exactly one edge incident with $S$ in $G'$.  Hence as $\deg_{G'}(S) \ge 3$, if $x$ has no $S$-external private neighbour, then $|Z|\ge 2$.
  
  Suppose $G' = G_1 \bbox G_2$ is a Cartesian product graph but some vertex in $S$  has no $S$-external private neighbour.  Then $x$ has no $S$-external private neighbour and it follows that $|Z|\ge 2$.  For $i\in\{1,2\}$, let $z_i$ be distinct elements of $Z$ and let $w_i$ be common neighbour of $x$ and $z_i$ which is not adjacent to any vertex of $S-\{x,z_i\}$. For a given $i$, $U_i=(S-\{x\})\cup \{w_i\}$ is a $\gamma$-set of $T$ in $G'$ adjacent to $S$.  Let $e_i= SU_i$.  If $e_1$ corresponds to an edge in $G_1$ and $e_2$ corresponds to an edge in $G_2$, note that  $e_1\bbox e_2$ is 4-cycle which is a subgraph of $G'$.  This contradicts that the vertices $a$ and $b$ in the statement of Proposition~\ref{prop:4cycle} are unique.  Hence we may assume each $e_i$ is in $G_1$. By Theorem~\ref{thm:gt:con}, $G_1$ and $G_2$ are connected. Then there exists an edge $e$ incident with $S$ corresponding to an edge in $G_2$. Then for each $i$, $e_i\bbox e$ is 4-cycle which is a subgraph of $G'$.  By Proposition~\ref{prop:4cycle}, $e$ corresponds to a swap of some $z\in Z$, $z=z_1$. Hence if $e=SW$, there exists a vertex $a\notin S$ so that $a\sim z_1$ and $W=(S-\{z_1\})\cup \{a\}$ is $\gamma$-set of $T$ in $G'$. It follows from Observation~\ref{basic} and that $z_1$ has an $S$-external private neighbour that $a\notin w_1$. Consider the 4-cycle $e_2\bbox e$.  By Proposition~\ref{prop:4cycle}, $V=S-\{x,z_1\}\cup \{w_2,a\}$ is a $\gamma$-set of $T$ with no neighbour of $w_1$, a contradiction.  Hence vertex in $S$  has an $S$-external private neighbour.

Conversely, suppose every vertex in $S$ has an $S$-external private neighbour.  By Lemma~\ref{k=1}, $k=1$ and it follows that, $y_1$ is the $S$-external private neighbour of $x$.  Rooting $T$ at $x$, we every descendant of $y_1$ in $L_1$ has at least two $L_1$-external private neighbours.  Hence every descendant of $y_1$ in $S$ has at least two $S$-external private neighbours. By Lemma~\ref{unique}, the only dominating set of $T$ containing $y_1$ is $L_1$ and hence for every $\gamma$-set $D\ne L_1$ of $T$ every descendant of $y_1$ in $S$ has at least two $D$-external private neighbours and can not be swapped. It follows $x$ is in every $\gamma$-set  which corresponds to a vertex in $G'$. Then  $|Z| = \deg_{G'}(S) \ge 3\ge 2$, and each $z_i\in Z$ has a distinct common neighbour with $x$.  Each $z_i$ is in a distinct branch of $T$, and swaps occur in each such branch independently.  Hence, $G'$ is a Cartesian product graph.
\end{proof}

Note that in the proof of the converse of the previous proposition the line, $\deg_{G'}(S) \ge 3$, $|Z| = \deg_{G'}(S) \ge 3\ge 2$.  For the result we require that $|Z|\ge 2$. Hence when a stem in the $\gamma$-graph has fewer non-leaf neighbours, we obtain the result in one direction. 

\begin{prop} \label{prop:gt:spn}
Let $G$ be a $\gamma$-graph of a tree $T$, $S$ be a stem of $G$, $L_1, L_2, \ldots,$ $L_k$ be the leaves of $S$, and $G' = G - \{L_1, L_2, \ldots, L_k\}$.  If $\deg_{G'}(S) \ge 2$, and every vertex in $S$ has an $S$-external private neighbour, then $G'$ is a Cartesian product graph.
\end{prop}

We immediately obtain the following result demonstrating where leaves cannot be attached to Cartesian product graphs.

\begin{Corollary}
Let $G$ be a $\gamma$-graph of a tree $T$, $S$ be a stem of $G$, $L_1, L_2, \ldots, L_k$ be the leaves of $S$, and $G' = G - \{L_1, L_2, \ldots, L_k\}$.  If $\deg_{G'}(S) \ge 3$, we can express $G' = G_1 \bbox G_2 \bbox \cdots \bbox G_n$, where each $G_i$ is `prime', and $S = (V_1, V_2, \ldots, V_n)$, then each $V_i$ is a leaf in $G_i$.
\end{Corollary}

\begin{proof}
By Lemma~\ref{lem:gt:stem}, 
there exists  vertices $x\in S$ and $y_1, y_2, \ldots, y_k\in V(T)-S$ so that $x\sim y_i$ and $L_i = (S - \{x\}) \cup \{y_i\}$ for $i=1, 2, \ldots k$. By Proposition~\ref{prop:gt:cart-spn}, $x$ has an $S$-external private neighbour and it follows from Lemma~\ref{k=1}, that $k=1$ and $y_1$ is the only $S$-external private neighbour of $x$. Rooting $T$ at $x$, and noting every descendant of $y_1$ in $L_1$ has at least two $L_1$-external private neighbours, it follows that $x$ is in every $\gamma$-set  which corresponds to a vertex in $G'$. 

Recalling that $T$ is rooted at $x$ let $Z$ be the set of vertices that can be swapped from $S$ in $G'$; each $z\in Z$ has a distinct common neighbour with $x$ and each $z\in Z$ is in a distinct branch of $T$.  Swaps occur in each such branch independently.  Therefore, if $G' = G_1 \bbox G_2 \bbox \cdots \bbox G_n$, where each $G_i$ is `prime', then without loss of generality each $z\in Z$ swap occurs in a different $G_i$.  Hence $V_i$ is a leaf in $G_i$ as only one swap corresponding to an edge in $G_i$ can occur.
\end{proof}

Let $\alpha^\dagger(G, v)$ denote the maximum number of edges incident with $v$ with no pair of edges part of the same 4-cycle.   We conclude by providing two additional tools to determine if graphs with stems are $\gamma$-graphs.  

\begin{Lemma}\label{alpha}
Let $G$ be a $\gamma$-graph of a tree $T$, $S$ be a stem of $G$, $L_1, L_2, \ldots, L_k$ be the leaves of $S$, and $G' = G - \{L_1, L_2, \ldots, L_k\}$.  If $\deg_{G'}(S) > 2 \alpha^\dagger(G', S)$, then $G'$ is a Cartesian product graph.
\end{Lemma}

\begin{proof}
Suppose $G'$ is not a Cartesian product graph.  If $L_j = (S - \{x\}) \cup \{y_j\}$ (Lemma~\ref{lem:gt:stem}), then it follows from Proposition~\ref{prop:gt:cart-spn}, that $x$ has no $S$-external private neighbour.  By  Proposition~\ref{prop:4cycle} at $x$ can be swapped from $S$ in $G'$ at most $\alpha^\dagger(G', S)$ times, so $\deg_T(x) \le k + \alpha^\dagger(G', S)$.  Let $\{x\} \cup Z$ be the set of vertices that can be swapped from $S$.  Then $\deg_{G'}(S) = (\deg_T(x)-k)+|Z|$. It can be shown, as in previous proofs, each $z\in Z$ has a distinct common neighbour with $x$ in $G'$, so $|Z| \le (\deg_T(x)-k)$ and hence $\deg_{G'}(S) \le 2 (deg_T(x)-k)\le 2 \alpha^\dagger(G', S)$, which is a contradiction.  Hence, $G'$ is a Cartesian product graph.
\end{proof}

\begin{prop}
Let $G$ be the $\gamma$-graph of a tree $T$, $S$ be a stem of $G$, $L_1, L_2, \ldots,$ $L_k$ be the leaves of $S$, and $G' = G - \{L_1, L_2, \ldots, L_k\}$.  If $\deg_{G'}(S) = 2$, $S$ is in the 4-cycle $PQRS$, and every vertex adjacent to $Q$ is in a 4-cycle with $P$ or $R$, then $G'$ is a Cartesian product graph.
\end{prop}

\begin{proof}
Suppose $G'$ is not a Cartesian product.  Let $L_j = (S - \{x\}) \cup \{y_j\}$ (Lemma~\ref{lem:gt:stem}), then it follows from Proposition~\ref{prop:gt:cart-spn}, that $x$ has no $S$-external private neighbour.  It follows from Lemma~\ref{alpha}, that $\deg_{G'}(S) \le 2 \alpha^\dagger(G', S)$ and hence $\alpha^\dagger(G', S) = 1$.  Rooting $T$ at $x$ each descendant of $y_j$ in $S$ has at least two $L_j$-external private neighbours, each descendant of $y_j$ in $S$ is in every $\gamma$-set of $T$.  Delete each $y_j$ from $T$ and let $T'$ be the component containing $x$.  Then $G' \cong T'(\gamma)$.

 As $\alpha^\dagger(G', S) = 1$, it follows from Proposition~\ref{prop:4cycle} only one swap in $G'$ can use $x$, so $\deg_{T'}(x) = 1$.  If $z$ is the other swap possible from $S$, then $x$ and $z$ have a common neighbour $w$, and $pn(z, S) = \{v\}$.  Hence we may assume without loss of generality that the neighbours of $S$ are $P = (S - \{z\}) \cup \{v\}$ and $R = (S - \{x\}) \cup \{w\}$. By Proposition~\ref{prop:4cycle}, $Q=(S - \{x,z\}) \cup \{w,v\}$.  Let $T_1$ be the component of $T' - zw$ containing $w$ and $T_2$ be the component of $T' - zw$ containing $z$.  Let $T_1+T_2$ be the disjoint union of $T_1$ and $T_2$. 
 
Suppose $A$ is a $\gamma$-set $T'$ but not a $\gamma$-set of $T_1+T_2$.   If $w \notin N_{T_1+T_2}[A]$, then $x \notin N_{T'}[A]$, which contradicts that $A$ is a $\gamma$-set of $T'$. Hence, we have $N_{T_1+T_2}[A]=V(T')-\{z\}$ and $w \in A$.  As $v$ is dominated by $A$, but $v$ is not in $A$, there exists a neighbour of $v$, $t\in A$. As $pn(z, S) = \{v\}$, it follows that $O=(S-\{x, z\})\cup \{w, t\}$ is a $\gamma$-set of $T'$. Hence, $O = (Q - \{v\}) \cup \{t\}$ is a $\gamma$-set of $T'$ with $v\sim t$ and $z \in pn(w, O)$ and therefore $O$ is adjacent to $Q$.  However, by Proposition~\ref{prop:4cycle}, $O$ does not form a 4-cycle with $P$, as $(O - \{w\}) \cup \{x\}$ is not a $\gamma$-set of $T'$, and $O$ does not form a 4-cycle with $R$, as $(R- \{v\}) \cup \{t\}$ is not a minimal dominating set ($v\notin R$).  Hence, every $\gamma$-set of $T'$ is a $\gamma$-set of $ T_1 + T_2$.  Further, since $x$ or $w$ is in every $\gamma$-set in $G'$, no swap in $T'$ uses the edge $zw$.  Therefore, $G' \cong T'(\gamma) \cong (T_1 + T_2)(\gamma) \cong T_1(\gamma) \bbox T_2(\gamma)$, which is a contradiction.
\end{proof}

\end{document}